\documentclass[10pt,regno,numberwithinsect]{amsart}
\usepackage{amsmath,amsthm,amsfonts,amssymb,amscd,overpic,mathrsfs}
\usepackage{dsfont}
\usepackage{amsthm}
\usepackage{amssymb}
\usepackage{graphicx}
\usepackage{amstext,amsthm,amsfonts,amsmath}
\usepackage[T1]{fontenc}
\usepackage{enumerate}
\usepackage{tikz}
\usepackage{pgfplots, pgfplotstable}
\usepackage[font=scriptsize]{caption}
\usepackage{indentfirst}
\usepackage{subfigure}

\newcommand{\ZZ}{{\mathbb{Z}}}

\newcommand{\bX}{{\mathbb{X}}}

\newcommand{\mcalB}{{\mathcal{B}}}

\newcommand{\eqdef}{\stackrel{\scriptscriptstyle\rm def}{=}}
\definecolor{Red}{cmyk}{0,1,1,0}

\definecolor{verde}{cmyk}{1,0,1,0}

\definecolor{azul}{cmyk}{1,1,0,0}

\begin{document}

\title[Non-hyperbolic Iterated Function Systems]{Non-hyperbolic Iterated Function Systems:
 semifractals and the chaos game}

\author[L. J.~D\'\i az]{Lorenzo J. D\'\i az}
\address{Departamento de Matem\' atica PUC-Rio, Marqu\^es de S\~ao Vicente 225, G\'avea, Rio de Janeiro 22451-900, Brazil}
\email{lodiaz@mat.puc-rio.br}

\author[E. Matias]{Edgar Matias}
\address{Departamento de Matem\' atica, ICMC-USP S\~ ao Carlos-SP, Brasil}
\email{edgarmatias9271@gmail.com}

\begin{abstract}
We consider iterated functions systems (IFS) on compact metric spaces and introduce the concept of  
 \emph{target sets}. Such sets have very rich dynamical properties and  play  a similar role as semifractals 
 introduced by Lasota and Myjak do for regular IFSs.  
  We study 
 sufficient conditions which guarantee that the closure of the target set is a local attractor for the IFS. As a corollary, 
 we establish necessary and sufficient conditions for the IFS having a global attractor. 
  We give an example of a non-regular IFS whose target set is nonempty,  showing that our approach gives rise to a ``new class'' of semifractals.
 Finally, 
 we show that random orbits generated by IFSs draws target sets that are ``stable''.
\end{abstract}

\begin{thanks}{This paper is part of the PhD thesis of EM (PUC-Rio) supported by CAPES.
LJD is partially supported by CNPq and CNE-Faperj.
The authors warmly thank P. Barrientos and K. Gelfert for their useful comments on this paper.  
}\end{thanks}
\keywords{Barnsley-Hutchinson operator, chaos game,
Conley and strict attractors, iterated function system, semifractals}
\subjclass[2000]{37C70,
 28A80,
 47H10.
 }

\date{}
\newtheorem{mteo}{Theorem}
\newtheorem{mcoro}{Corollary}
\newtheorem{mprop}{Proposition}

\newtheorem{teo}{Theorem}[section]
\newtheorem{prop}[teo]{Proposition}
\newtheorem{lema}[teo]{Lemma}
\newtheorem{schol}[teo]{Scholium}
\newtheorem{coro}[teo]{Corollary}
\newtheorem{defi}[teo]{Definition}
\newtheorem{ques}[teo]{Question}
\newtheorem{example}[teo]{Example}
\newtheorem{remark}[teo]{Remark}
\newtheorem{nota}[teo]{Notation}
\newtheorem{claim}[teo]{Claim}
\newtheorem{fact}[teo]{Fact}

\numberwithin{equation}{section}

\maketitle


\section{Introduction}


Consider 
finitely many contractions $f_{1},\dots, f_{k}$ defined on a complete metric space $X$ and 
a sequence  of 
independent and identically distributed
(i.i.d.) random variables $\bX_{n}$ 
with range $\{1,\dots, k\}$. It is well known 
that for every $x$, with probability $1$, the {\emph{tail sequence of sets}} 
$\{f_{\bX_{n-1}}\circ \dots \circ f_{\bX_{0}}(x)\colon n\geq \ell\}_{\ell}$
converges to the unique compact set $K$
satisfying 
$$
K=\bigcup_{i=1}^{k} f_{i}(K),
$$
called the {\emph{Hutchinson attractor}}. The existence and uniqueness of the latter was proved in \cite{Hu}. 

In this paper, we consider a setting where the maps 
$f_{1},\dots, f_{k}$ are just continuous maps defined on a compact metric space
and we translate several results of the so-called
``Hutchinson theory'',  as the ones described above, to this
non-contracting setting.
The finite family of continuous maps $f_{1},\dots, f_{k}$ is called an \emph{iterated function system} (in what follows, an IFS) and
denoted by $\mathrm{IFS}(f_{1},\dots,f_{k})$. An IFS is called \emph{hyperbolic} if all of its maps are (uniform) contractions, 
and \emph{non-hyperbolic} otherwise.

A key ingredient in this study is the
so-called
 {\emph{Barnsley-Hutchinson operator}} (in the sequel, the $\mathrm{BH}$-operator for short). Considering  the 
 hyperspace  $\mathscr{H}(X)$ 
 of non-empty compact subsets  of $X$ 
 endowed with the Hausdorff distance, the $\mathrm{BH}$-operator  associated to  $\mathrm{IFS}(f_{1},\dots,f_{k})$ is defined
 by
\begin{equation}\label{e.BH}
\mathcal{B} \colon \mathscr{H}(X) \to \mathscr{H}(X), \quad
\mathcal{B} (A) \eqdef  \bigcup_{i=1}^{k} f_{i}(A).
\end{equation}
 An important aspect of the Hutchinson theory concerns 
 the study of the dynamics
of the $\mathrm{BH}$-operator and its fixed points.
 The {\emph{Hutchinson attractor}} of a hyperbolic IFS is the unique fixed point of the $\mathrm{BH}$-operator. Moreover, in this case, this attractor 
 is global, see \cite{Hu, fractalsevery}.

A first important class of non-hyperbolic IFSs considered in this study is the one
of \emph{contracting on average} IFSs, see \cite{BarElton}.
  For this class, when studying IFSs from the probabilistic point of view, 
  the results in \cite{BarElton} fully recovered the results about stationary measures
  in  \cite{Hu}.
 On the other hand, from the topological 
   point of view, the scenario may change substantially.
The dynamics of the $\mathrm{BH}$-operator of a contracting on average IFS is, in general, very different from the one in the hyperbolic case. 
For instance, concerning compact fixed points, it may have several ones or none at all. 
Nevertheless, by \cite{Lasota}, it always possesses one  (non-necessarily compact) fixed point
with  very rich dynamical properties which is called its \emph{semifractal} and,
in a certain sense, generalizes 
the Hutchinson attractor.  We observe that semifractals 
were introduced for so-called {\emph{regular}} IFSs (see Section \ref{regularifs} for the precise definition), which is a class of IFSs much more general than the contracting on average one.
 A semifractal is characterised by the following two properties:
 \begin{enumerate}
 \item[{\bf (SF1)}] it is the {\emph{minimum fixed}} point  of the 
 $\mathrm{BH}$-operator (i.e., it is contained in any other fixed point of the $\mathrm{BH}$-operator) 
  and 
 \item[{\bf (SF2)}] it attracts each of its compact subsets under the action of the 
  $\mathrm{BH}$-operator. 
  \end{enumerate}

Another class of non-hyperbolic IFSs was considered by Edalat in
  \cite{Edalat}, also fully recovering aspects of the Hutchinson theory. In what follows, we assume that $X$ is compact.
The $\mathrm{IFS}(f_{1},\dots,f_{k})$ is called \emph{weakly hyperbolic}
if it satisfies the following topological condition:
\begin{equation}\label{e.edalat}
\lim_{n\to \infty}\mathrm{diam}\,\big( f_{\omega_{0}}\circ\cdots\circ f_{\omega_{n}}(X)\big)= 0
\end{equation}
 for every $ \omega=\omega_0\omega_1\omega_2\dots \in\Sigma_k\eqdef\{1,\dots,k\}^{\mathbb{N}}$.
In \cite{Edalat} it is proved that the
$\mathrm{BH}$-operator
of a weakly hyperbolic  IFS 
  has a global attracting fixed point. Note that every hyperbolic IFS is weakly hyperbolic.
Observe that by a recent result in \cite{Nowak}, after a change of metric, every weakly hyperbolic system can be made hyperbolic. 

Bearing the discussion above in mind, we introduce 
 the subset $S_{\mathrm{wh}}\subset\Sigma_k$ 
 of {\emph{weakly hyperbolic sequences}}
 defined by
\begin{equation}\label{e.stweakly}
S_{\mathrm{wh}}\eqdef\big\{\omega\in \Sigma_k\colon 
 \displaystyle\lim_{n\to \infty}\mathrm{diam}\big(f_{\omega_{0}}\circ\cdots\circ f_{\omega_{n}}(X)\big)=0\big\}.
\end{equation}
Note that if the IFS is weakly hyperbolic then $S_{\mathrm{wh}}=\Sigma_k$. 
In this paper, we will recover parts of the Hutchinson theory in  ``truly'' non-hyperbolic settings where we only require 
that the set $S_\mathrm{wh}$ is non-empty. In this analysis, a key object is the so-called 
\emph{target set} defined as follows.
Associated to the set $S_{\mathrm{wh}}$  we consider
the \emph{coding map} $\pi\colon S_{\mathrm{wh}}\to X$ that projects $S_{\mathrm{wh}}$ into the phase space $X$
defined by
\begin{equation}\label{e.pi}
\pi\colon S_\mathrm{wh} \to X ,
\quad
\pi(\omega)\eqdef\lim_{n\to \infty}f_{\omega_{0}}\circ\cdots\circ f_{\omega_{n}}(p),
\end{equation}
where $p$ is any point of $X$.
By definition of the set $S_{\mathrm{wh}}$, this limit always exists 
and is independent of $p\in X$.  The set of weakly hyperbolic sequences generates the {\emph{target set}} defined by
\begin{equation}
\label{e.targetlzo}
A_{\mathrm{tar}}\eqdef 
 \pi(S_{\mathrm{wh}}).
 \end{equation}
 Our results show that the target set encodes crucial aspects of the dynamics of the 
 $\mathrm{BH}$-operator. Let us now briefly discuss these results, the precise statements can be found in
 Section~\ref{s.mainresults}.

 In what follows we always assume that $S_{\mathrm{wh}}\ne \emptyset$.
 Under this assumption,
 we prove that the  closure of the target set satisfies the two characteristic properties
 {\bf{(SF1)}} and  {\bf{(SF2)}} of semifractals, see Theorem~\ref{semifractal}. We also give an example of a non-regular IFS whose target set is non-empty.   This shows that our approach gives rise to a ``new class'' of semifractals.

Let us observe that, in general, the closure of the target set may fail to be a global attractor. 
In Theorem~\ref{mt.local}, we study necessary and sufficient conditions guaranteeing that
 the closure of the  target set 
  is a ``local'' attractor. As a consequence,  
we state a sufficient and necessary  condition assuring that the 
$\mathrm{BH}$-operator  has a global attractor, generalising 
Edalat's result \cite{Edalat} mentioned above.   
Moreover, we show that if the $\mathrm{BH}$-operator has a unique fixed point, then it is necessarily a global attractor, 
see Theorem~\ref{mt.att}.

Finally, we investigate under which conditions {\emph{a random orbit generated by an IFS draws the target.}}
For a given $\mathrm{IFS}(f_{1},\dots,f_{k})$, consider a sequence 
of i.i.d. random variables $\bX_{n}$ with range in $\{1,\dots, k\}$ and for $i\in \{1,\dots, k\}$  let 
$
p_{i}\eqdef \mbox{prob}(\bX_{n}=i)>0$.
The procedure of choosing at random $i$ with probability $p_{i}$ and 
moving from $x$ to $f_{i}(x)$ is an algorithm called the  \emph{(probabilistic) chaos game}.
In this way,
 for each $x\in M$, we get a Markov chain 
 $\ZZ_{n}\eqdef f_{\bX_{n-1}}\circ\dots\circ f_{\bX_{0}}(x)$ 
  moving through the set $X$.
For hyperbolic IFSs, a consequence of standard ergodic theorems (see for instance \cite{Breiman})  is that the tail of the Markov chain $\{\ZZ_{n}\colon n\geq \ell\}_\ell$
 converges in the Hausdorff distance to the 
Hutchinson attractor with probability one.

An standard way of taking an i.i.d. sequence with range in $\{1,\dots, k\}$ is the following. Consider the space 
$\Sigma_{k}$ as a probability space endowed with the product $\sigma$-algebra and the product measure 
$\mathbb{P}=\nu^{\mathbb{N}}$, where $\nu$ is the probability measure in $\{1,\dots,k\}$ given
by $\nu (i)=p_i$.
Then the natural projections
$\bX_{n}(\omega)=\omega_{n}$ is a sequence of i.i.d. random variables. An advantage of considering $\Sigma_{k}$ as the sample space of the random variables $\bX_{n}$ is that we can use the structure of $\Sigma_{k}$ to improve the understanding and results
about the chaos game, and also to get more elementary proofs.
 The results in \cite{Lesniak} and Theorem~\ref{mt.jogodocaos} illustrate well these assertions. Let us explain this a bit more precisely.


We show in Theorem~\ref{mt.jogodocaos} that if the closure of the target set is 
\emph{stable} (see Section \ref{thechaosgamee} 
for the definition) then the subset $\mathcal{D}$ 
of \emph{disjunctive sequences} of $\Sigma_k$ (see Section~\ref{thechaosgamee} for details)
 is 
such that for every $x\in X$ 
the sequence of tail sets satisfies
$$
\lim_{\ell\to \infty}
\{f_{\omega}^{n}(x)\colon n\geq \ell\}_\ell =
\mathrm{closure}(A_{\mathrm{tar}}),
\quad\text{where}\quad
f_{\omega}^{n}(x)\eqdef f_{\omega_{n-1}}\circ \dots\circ f_{\omega_{0}}(x),
$$
for every $\omega\in \mathcal{D}$.
We observe that it follows from its definition that the set $\mathcal{D}$ does not depend on $x$ and
satisfies 
$\mathbb{P}(\mathcal{D})=1$.

The procedure of choosing $\omega$ in a
prescribed set and considering the deterministic orbit
$f_{\omega}^{n}(x)$ is called the \emph{deterministic chaos game}. 
Notice that in this study there is no stationary process associated,  hence it is not possible 
to use ergodic theorems (as in the classical theory) or  Markov operators (as in \cite{Myjak,Szarek}) to prove that deterministic orbits draw the attractor.
It is worth mentioning that, remarkably, even in the probabilistic chaos game it is possible to avoid such probabilistic techniques, see for instance \cite{Vincechaos, Rypka, Barrientos}.
Here Theorem~\ref{mt.jogodocaos} contributes to a series of  works about the deterministic chaos game, see \cite{Lesniak, Lesniaknon, Barrientos},
the novelties of our result 
are discussed in Section \ref{thechaosgamee}.

\subsection*{Organisation of the paper} 
In Section \ref{s.mainresults} we state precisely the main results in
this paper. The theorems are presented in  Section \ref{s.mainresults} and proved
in Section \ref{s.proofs}, which consists of a section of preliminary general results and  four subsections
  dedicated to prove these theorems.
In Section 
\ref{s.examples}, we present some examples illustrating our results and hypotheses.  
 
\medskip\noindent\textbf{Standing hypotheses.}
Throughout this paper, we always consider 
a complete metric space $(X,d)$, an $\mathrm{IFS}(f_1,\dots,f_k)$ of continuous maps defined on $X$, 
 and its corresponding $\mathrm{BH}$-operator $\mathcal{B}$. 
 Thus, these hypotheses are omitted in the statement of our theorems.

\medskip
Convergence of compact sets is always considered with respect to
the {\emph{Hausdorff distance} $d_H$ defined as follows.
Given a point $x\in X$ and a set $A\subset X$,  the distance between $x$ and $A$
 is defined
  by 
 $$
 d(x,A)\eqdef \inf\{d(x,a)\colon a\in A\}.
 $$
 Recall that 
 $\mathscr{H}(X)$ is the set of
 non-empty compact subsets  of $X$. 
 The {\emph{Hausdorff distance}} between two compact sets $A,B\in\mathscr{H}(X)$ is defined by 
\begin{equation}
\label{e.HD}
 d_{H}(A,B)\eqdef \max\{h_{s}(A,B),h_{s}(B,A)\},
 \quad
 \mbox{where} 
\quad
 h_{s}(A,B)\eqdef \sup_{a\in A}d(a,B).
 \end{equation}
 We observe that  $(\mathscr{H}(X),d_{H})$ is a compact metric space, see \cite{fractalsevery}.
 
 \section{Main results}
\label{s.mainresults}

\subsection{Target sets and semifractals}\label{regularifs}
Our first result shows that the closure of the target set satisfies the properties 
{\bf (SF1)} and {\bf (SF2)}
characterising the semifractals of regular IFSs introduced by Lasota  and  Myjak in \cite{Lasota}.
An $\mathrm{IFS}(f_{1},\dots, f_{k})=\mathfrak{F}$ 
 is said to be \emph{regular} if there
 are numbers $1\le i_{1}<i_{2}<\dots< i_{\ell} \leq k$ such that 
 $\mathrm{IFS}(f_{i_{1}},\dots,f_{i_{\ell}})=\mathfrak{F}'$ has a global attractor. The global attractor of $\mathfrak{F}'$ is called a 
{\emph{nucleus}} of $\mathfrak{F}$ (an IFS may have several
nuclei).
It is proved in \cite{Lasota} that the $\mathrm{BH}$-operator of any regular IFS $\mathfrak{F}$ 
has a fixed point,
called  a \emph{semifractal}, satisfying the properties
{\bf (SF1)} and {\bf (SF2)}.
%
In the next theorem, replacing the regularity of the IFS by the condition $S_{\mathrm{wh}}\neq \emptyset$, we
 recover these properties. Recall the definition of the target set 
$A_{\mathrm{tar}}$ in
 \eqref{e.targetlzo}.

\begin{mteo}\label{semifractal}
Consider an IFS with $S_{\mathrm{wh}}\neq \emptyset$.
Then the following holds:
\begin{itemize}
\item [(1)]
The set $\overline{A_{\mathrm{tar}}}$ is the minimum fixed point of 
$\mathcal{B}$.

\item [(2)]
For every compact set $K\subset \overline{A_{\mathrm{tar}}}$ it holds 
$$
\lim_{n\to \infty} d_{H}(\mathcal{B}^{n}(K), \overline{A_{\mathrm{tar}}})=0.
$$
\end{itemize}
\end{mteo}

\begin{remark}{\em{The regularity of the IFS does not imply that $S_{\mathrm{wh}}\neq \emptyset$ 
and vice-versa. Example \ref{ex.nonregular} provides an IFS with $S_{\mathrm{wh}}\neq \emptyset$   which is not regular.
Hence, our approach provides  a ``new class'' of semifractals. Since 
minimum fixed points are unique, it follows that for regular IFSs with  $S_{\mathrm{wh}}\neq \emptyset$ the semifractal
(as defined by Lasota and Myjak)  is the closure of the target set.}}
\end{remark}

\begin{remark}\emph{Zorn's lemma provides a
\emph{minimal fixed point} for the $\mathrm{BH}$-operator (i.e.,
a fixed point that does not contain properly another fixed point).
However, as mentioned above,  minimum fixed points may not exist.
Note that, by definition, a minimum fixed point is minimal, but the converse is not true in general.
Above we have seen two different contexts leading to minimum fixed points: regularity of the IFS and 
the existence of some weakly hyperbolic sequence.
It is natural to ask for a unifying condition  guaranteeing the existence of a minimum fixed 
point of the $\mathrm{BH}$-operator. }
\end{remark}


\subsection{Attractors of the Barnsley-Hutchinson operator}\label{strictconley}
To  state our  next results precisely we need the following definition.

\begin{defi}[Strict and Conley attractors]
{\em{
A compact set  $A\subset X$ is a  \emph{strict attractor}  
of the $\mathrm{BH}$-operator
 if there is an open neighbourhood
 $U$ of $A$ such that 
 $$
\lim_{n\to\infty} \mathcal{B}^{n}(K)= A
 \quad
 \mbox{for every  compact set $K\subset U$.}
 $$
The \emph{basin of attraction} of $A$ is the largest open neighbourhood of $A$ for which the above  property holds. A strict attractor whose basin of attraction is the whole space is  a 
{\emph{global attractor.}}

  A compact set $S\subset X$ is a \emph{Conley attractor} of  the $\mathrm{BH}$-operator
   if there exists  an open neighbourhood $U$ of $S$ such that
 $$
 \lim_{n\to \infty} \mathcal{B}^{n}(\overline{U})=S.
 $$}}
 \end{defi}
 
 \begin{remark}{\em{
The continuity of the $\mathrm{BH}$-operator    implies that Conley  
  and strict attractors are both fixed points of  $\mathcal{B}$.
Note also that strict attractors are Conley attractors but the converse is not true in general.}}
 \end{remark}
 
 The next theorem shows that the target set  plays a key role in the study of strict attractors
 of the $\mathrm{BH}$-operator
 for IFSs with $S_{\mathrm{wh}}\neq \emptyset$.

%

 \medskip

 \begin{mteo}\label{mt.local}
 Consider an IFS with $S_{\mathrm{wh}}\neq \emptyset$. 
Then the following  holds:
 \begin{itemize}
 \item[(1)]
 The  $\mathrm{BH}$-operator 
has at most one strict attractor. Moreover, if such a strict 
attractor exists then 
it is equal to $\overline{A_{\mathrm{tar}}}$.
 \item[(2)]
  $\overline{A_{\mathrm{tar}}}$ is a 
 Conley attractor if and only if it is a strict attractor. 
 \end{itemize}
 \end{mteo}
%

Let us observe that in \cite{Vince} Vince 
obtained a result similar to Theorem \ref{mt.local} for  a completely different class of non-hyperbolic IFSs.
He shows that an
IFS consisting of M\"obius maps (M\" obius IFSs) has at most one strict attractor
and that a compact set is a Conley attractor if and only if it is a strict attractor. Note that 
for M\"obius IFSs the set $S_{\mathrm{wh}}$ is empty.

The next theorem generalises Edalat's results in \cite{Edalat} in two ways:  
(i) it applies to IFSs which are not weakly hyperbolic,
(ii)  it provides not only a sufficient condition (as in \cite{Edalat}) but also a necessary condition for the existence of a global attractor. To state the theorem 
consider 
  the set
  $$
X^* \eqdef
\bigcap_{n\geq 0} \mathcal{B}^{n}(X)
$$
which is a fixed point of  $\mathcal{B}$ (see Lemma \ref{l.p.existence}).
Note that the set $X^{*}$ is the \emph{maximum fixed point}
 of the map $\mathcal{B}$ (i.e.,
every fixed point of $\mathcal{B}$ is contained in $X^{*}$).

\begin{mteo}\label{mt.att}
 Consider an IFS with $S_{\mathrm{wh}}\neq \emptyset$. 
 The following three assertions are equivalent: 
\begin{itemize}
 \item[(1)] 
$\overline{A_\mathrm{tar}}=X^*$,
 \item[(2)]
the $\mathrm{BH}$-operator has a unique fixed point,
\item[(3)]
$X^{*}$ is a global attractor of the $\mathrm{BH}$-operator.
\end{itemize}
\end{mteo}

The Hutchinson theory, the hyperbolic as well as the non-hyperbolic one,
has a topological and a probabilistic point of view. Despite 
their different nature,
these two points of view ``talk to each other'' through the coding map, which plays a key role in both 
approaches. Let us illustrate this ``talk''. Starting with the probabilistic side,  
it is well known, see Letac \cite{Letac}\footnote{Note that the terminology in \cite{Letac} is
different. We translated the results to our terminology.}, that if $\mathbb{P}(S_{\mathrm{wh}})=1$, then 
the distribution (``probabilistic image'') of the coding map is an attracting\footnote{Under the action of
the Markov operator associated with the Markov chain.} stationary measure of the Markov 
chain $\omega\mapsto \ZZ_{n}(\omega)\eqdef f_{\omega}^{n}(x)$, for every $x$. In other words, if 
$S_{\mathrm{wh}}$ is big in the probabilistic sense, one of the main results of the hyperbolic Hutchinson 
theory  (uniqueness of stationary measures)
 is
recovered.  Considering now the topological side,
item (1) of Theorem \ref{mt.att} shows that if the 
target set
(i.e, the topological image of the coding map) is topologically big
 then the $\mathrm{BH}$-operator has a global attractor. 
Note that when
$\mathbb{P}(S_{\mathrm{wh}})=1$, it is clear that $\overline{A_{\mathrm{tar}}}$ is the support of the unique stationary measure $m$,
which is given by $m(A)=\mathbb{P}(\pi^{-1}(A))$.
This discussion leads to the following question: 

\begin{ques}{\em{
Theorem~\ref{mt.att} claims that 
the target set is ``big'' if and only the $\mathrm{BH}$-operator has a global attractor. On the other hand, 
Letac's result \cite{Letac} only provides the sufficient condition {\emph{if $S_{\mathrm{wh}}$  is ``big'' then there is
an attracting stationary measure.}} Is the converse true? }}
\end{ques}

\begin{remark}\label{EDtarget}
\emph{The name \emph{target set}  for $ A_{\mathrm{tar}}$ is justified by the following characterisation 
\begin{equation}\label{e.characterizationAt}
 A_{\mathrm{tar}}=\{x\in X \colon \mbox{there is $\omega \in \Sigma_k$ with $\{x\}=\bigcap_n 
 f_{\omega_0}\circ\dots\circ f_{\omega_n}(X)$}\}.
\end{equation}
This follows observing that $\{f_{\omega_{0}}\circ \dots \circ f_{\omega_{n}}(X)\}$ is a sequence of nested compact sets.}
\emph{
Let us observe that for weakly hyperbolic IFSs, in \cite{Edalat} it is proved that 
\begin{equation}\label{edalat}
X^{*}=\bigcup_{\omega\in \Sigma_{k}} I_{\omega}, 
\quad \mbox{where $I_{\omega}\eqdef\bigcap_n f_{\omega_{0}}\circ \dots \circ f_{\omega_{n}}(X)$.}
\end{equation}
In particular, for weakly hyperbolic IFSs it holds $A_{\mathrm{tar}}=X^{*}$.}

{\em{
 It is worth mentioning that there are  non-weakly hyperbolic IFSs still satisfying
$A_{\mathrm{tar}}=  \overline{A_{\mathrm{tar}}}= X^{*}$, see Example~\ref{ex.ultimahora}. 
However, 
 since there are cases
 where 
$A_{\mathrm{tar}}\subsetneq \overline{A_{\mathrm{tar}}}= X^*$, see Example~\ref{ex.outromais},
 it is indispensable considering
the closure of the target set in item (1) of Theorem~\ref{mt.att}.}}
\end{remark}

%

\subsection{The chaos game}\label{thechaosgamee}
We now see how random iterations draw ``stable'' target sets. For that
we need some preliminary definitions.
Recall that 
the shift on $\Sigma_{k}$ is the map $\sigma\colon \Sigma_k \to \Sigma_k$ defined by $(\sigma(\omega))_n=\omega_{n+1}$.
 A sequence of $\Sigma_k$ whose orbit is dense in $\Sigma_{k}$, with respect to the shift map, is called \emph{disjunctive}. We denote by $\mathcal{D}$ the set 
 of disjunctive sequences. 
Note that,  as a consequence of Birkhoff's theorem, for every ergodic measure $\mu$ (with respect to  $\sigma$)
 with full support it holds $\mu(\mathcal{D})=1$.

 

A fixed point $K$ of the $\mathrm{BH}$-operator
is  \emph{stable} if for every open neighbourhood $V$ of $K$ there
is an open  neighbourhood $V_{0}$ of $K$ such that 
\begin{equation}
\label{e.stableBH}
\mathcal{B}^{n}(V_{0})\subset V \quad \mbox{for every}\quad n\geq 0.
\end{equation}
For instance,
the set $\overline{A_{\mathrm{tar}}}$ is stable if it is a Conley attractor
(but this is not a necessary condition, see Example~\ref{example1})
or if 
all the  maps of the IFS  are {\emph{non-expanding}} (i.e., all maps are Lipschitz with  constant less than or equal to $1$)\footnote{In the literature, these systems are also called {\emph{non-expansive.}}}. 
 
\begin{mteo} 
\label{mt.jogodocaos}
Consider  an IFS 
such that $\overline{A_{\mathrm{tar}}}$ is a stable fixed point of the $\mathrm{BH}$-operator.
 Then for every  $x\in X$ and every disjunctive sequence $\omega\in \Sigma_k$ we have
 $$
\lim_{\ell\to \infty} \overline{\{f_{\omega}^{n}(x)\colon n\geq \ell\}}= \overline{A_{\mathrm{tar}}}.
$$
\end{mteo}


We now compare Theorem~\ref{mt.jogodocaos} with some previous results in \cite{Lesniak,Barrientos,Lesniaknon}.

In \cite{Lesniak}
the authors consider an IFS$(f_{1},\dots, f_{k})$ 
 having a strict attractor  $K$ that is \emph{strongly fibred}, meaning that for every point $x\in K$ and every 
 neighbourhood $V$ of $x$, there is a finite word $c_{1}\dots c_{\ell}$ such that $f_{c_{\ell}}\circ \dots\circ
 f_{c_{1}}(K)\subset V$.  
 Under the strongly fibred condition,  it is proved that for every $x$ in the basin of attraction of $K$ it holds
 \begin{equation}\label{tail}
 \lim_{\ell\to \infty} \{f_{\omega}^{n}(x)\colon n\geq \ell\}=K, 
 \end{equation}
 for every disjunctive sequence $\omega$.
We observe that the closure of the target set is a fixed point strongly fibred but, in general, 
it is not a strict attractor, therefore the results in \cite{Lesniak} cannot be applied
since being stable is a weaker condition than being a strict attractor.
Even in the case when $\overline{A_{\mathrm{tar}}}$ is a strict attractor, Theorem~\ref{mt.jogodocaos}
provides further information:
the deterministic chaos game can start at \emph{any point} $x$ and not only on the basis of attraction of
the set
$\overline{A_{\mathrm{tar}}}$.

The deterministic 
chaos game for \emph{well-fibred quasi-attractors} is studied in  \cite{Barrientos}.
We refrain 
to give the precise definition of a  well-fibred quasi-attractor.  For our goals it is sufficient to say that every strong fibred 
fixed point of $\mathrm{BH}$-operator is also a well-fibred quasi-attractor. 
In \cite{Barrientos} it is proved that given any well fibred quasi-attractor $K$  then, 
 for every disjunctive sequence $\omega$, the tail sequence $\{f_{\omega}^{n}(x)\colon n\geq \ell\}_\ell$ converges to 
 $K$ for every point $x$ in the set 
 $$
 \mathcal{B}^{*}(K)\eqdef \{x\in X\colon \bigcap_{m\geq 0} \overline{\bigcup_{n\geq m}\mathcal{B}^{n}(\{x\})}=K \}.
 $$   
 Since
  the closure of the target set is strongly fibred, it is also  well-fibred and hence 
 the result in \cite{Barrientos} can be applied. In this case, Theorem \ref{mt.jogodocaos} 
 also provides additional information since, even if $A_{\mathrm{tar}}$ is stable, the set
 $\mathcal{B}^{*}(K)$
  may not be  
 the whole space, see Example \ref{example1}.

Finally, \cite{Lesniaknon} studies the deterministic chaos game for strict attractors of non-expanding IFSs.
As observed above, in this case the target set is stable (but not necessarily a strict attractor, see
Example~\ref{example1}) and hence Theorem~\ref{mt.jogodocaos} applies to cases which are not covered
by \cite{Lesniaknon}.

\section{Proofs of the Theorems}
\label{s.proofs}

\subsection{Preliminary results}
\label{ss.preliminary}
We start with two auxiliary results. The first one is an elementary lemma about fixed points of the $\mathrm{BH}$-operator
and the second one is an important lemma that will be used in the proofs of Theorems \ref{semifractal} and \ref{mt.local}. 

\begin{lema}\label{l.p.existence}
Consider  $A\in \mathscr{H}( X)$ such that  $\mathcal{B}(A)\subset A$. 
Then the set 
\begin{equation}
\label{e.conjestrela}
A^* \eqdef
\bigcap_{n\geq 0} \mathcal{B}^{n}(A)
\end{equation}
is a  fixed point of $\mathcal{B}$.
In particular, the set $X^{*}$ is a fixed point of $\mcalB$.
\end{lema}

\begin{proof}
We need the following
 well-known result whose proof is included for completeness.
\begin{claim}
 \label{c.l.semnome}
 Let $(L_{n})$ be a sequence of nested compact sets, i.e., with $L_{n+1}\subset L_n$.
 Then
 $$
 \lim_{n\to \infty}d_{H}(L_{n},L)= 0, \quad \mbox{where  $L\eqdef \bigcap_{n\geq0} L_{n}$.}
 $$
 \end{claim}
 
  \begin{proof}
The proof is by contradiction. If the claim is false there are
 $\epsilon>0$ and a subsequence $(n_{\ell})$, $n_\ell \to \infty$, such that 
 $d_{H}(L_{n_{\ell}},L)\geq \epsilon$ for all $\ell$. 
 Since $L\subset L_{n}$ for every $n$, 
 for each $\ell$
 there is a point $p_{\ell}\in L_{n_{\ell}}$ such that $d(p_{\ell},L)\geq \epsilon$. 
 By compactness, taking a subsequence if necessary, we can assume that 
 $p_{\ell}\to p$. 
 As $ (L_{n})$ is nested 
 it follows that $p\in L$, contradicting  that
 $d(p_{\ell},L)\geq \epsilon$ for all $\ell$. 
\end{proof}
Applying Claim \ref{c.l.semnome} to the sequence $L_n=\mathcal{B}^n(A)$ and using the continuity of $\mathcal{B}$, the lemma follows.
\end{proof}

\begin{lema}\label{dista}
Consider an IFS with $S_{\mathrm{wh}}\neq \emptyset$. Then 
for every compact set $K$ it holds 
$$
\lim_{n\to \infty} h_{s}(\overline{A_{\mathrm{tar}}}, \mathcal{B}^{n}(K))=0.
$$

\end{lema}
\begin{proof}
The proof is by contradiction. Assume 
 that there are a compact set $K\in \mathscr{H}(X)$ and a sequence $(n_\ell)$ such that
$h_{s}(\overline{A_{\mathrm{tar}}},\mathcal{B}^{n_{\ell}}(K))\ge \epsilon$ for every
$\ell$. 
Note that for each $\ell$ 
there is $p_{n_\ell}\in \overline{A_{\mathrm{tar}}}$ with $d(p_{n_\ell},\mathcal{B}^{n_\ell}(K))\ge \epsilon$. 
By compactness, we can assume that
$p_{n_{\ell}}\rightarrow p^{*}\in \overline{A_{\mathrm{tar}}}$ 
and that $\mathcal{B}^{n_{\ell}}(K)\rightarrow  \widehat{K}$ for some compact set $\widehat{K}$.
 We need the following elementary claim whose prove we present for completeness: 
 \begin{claim}
\label{l.era14}
 Consider sequences
 $(A_n)$ in $\mathscr{H}(X)$ and
 $(p_n)$ of points in $X$
 with
 $A_n\to A$ and $p_n\to p$ in the  
 Hausdorff distance $d_{H}$. 
 Then 
 $$
 d (p,A)=\displaystyle\lim_{n\rightarrow \infty} d(p_{n},A_{n}).
 $$
 \end{claim}
 
 \begin{proof}
 We use  the following ``triangular'' inequality:
 given a point $q$ and two compact sets $A$ and $B$ it holds
 $$
 d(q,A)\le d(q,B)+ d_H(A,B).
$$ 
Consider the sequences $(A_n)$ and $(p_n)$ in the claim. 
 Applying twice the ``triangular'' inequality
 above we get
$$
 d (p,A)\leq d(p,p_{n})+d(p_{n},A)\leq  d(p,p_{n})+d(p_{n},A_{n})+d_H(A_{n},A).
$$
By hypothesis,  $d(p,p_{n})\to 0$ and $d_{H}(A_{n},A)\to 0$.  
We conclude that 
\begin{equation}
 \label{eq.liminf}
d(p,A)\leq \liminf_{n} d(p_{n},A_{n}).
\end{equation}

Applying again twice the ``triangular'' inequality, we get
 $$ 
 d (p_{n},A_{n})\leq 
 d(p_{n},p)+d(p,A_{n})\leq d(p_{n},p)+d(p,A)+ d_H(A,A_{n}).
 $$
Arguing as above, this implies that
\begin{equation}
 \label{eq.limsup}
 \limsup_{n} d(p_{n},A_{n})\leq d(p,A).
 \end{equation}
 Equations \eqref{eq.liminf} and \eqref{eq.limsup} imply the claim.
 \end{proof}
 It follows from Claim \ref{l.era14} that
\begin{equation}
\label{e.acontradicao}
d(p^{*},\widehat{K})\geq \epsilon.
\end{equation}
We now derive a contradiction from this inequality.
By construction, there is 
 $\ell_{0}$ such that 
\begin{equation}
\label{e.byconstruction}
h_{s}(B^{n_{\ell}}(K),\widehat{K})<\frac{\epsilon}{3},
\quad
\mbox{for all $\ell\geq \ell_0$.}
\end{equation}
Take $q\in B_{\frac{\epsilon}{3}}(p^{*})\cap A_{\mathrm{tar}}$. By the definition of $A_{\mathrm{tar}}$ in
\eqref{e.characterizationAt},
there is a sequence $\omega=\omega_{0}\
\omega_{1}\ldots\in S_\mathrm{wh}$ such that 
$$
\bigcap_{n\geq 0} f_{\omega_{0}} \circ\cdots\circ f_{\omega_{n}}(X)=\{q\}.
$$
Therefore there is $m_{0}$  such that
$$
f_{\omega_{0}} \circ\cdots\circ f_{\omega_{m-1}}(K)\subset
f_{\omega_{0}} \circ\cdots\circ f_{\omega_{m-1}}(X)\subset B_{\frac{\epsilon}{3}}(p^{*})
\quad
\mbox{for every  $m\geq m_{0}$}.
$$
Since $f_{\omega_{0}} \circ\cdots\circ f_{\omega_{m-1}}(K)\subset  \mathcal{B}^{m}(K)$,
 we have 
$$
\label{e.largeenough}
\mathcal{B}^{n_{\ell}}(K)\cap B_{\frac{\epsilon}{3}}(p^{*})\neq \emptyset
\quad
\mbox{for every  $\ell$ big enough.}
$$
Taking now any $z\in \mathcal{B}^{n_{\ell}}(K)\cap B_{\frac{\epsilon}{3}}(p^{*})$ (large $\ell$) and
using 
equation
\eqref{e.byconstruction} we have
$$
d(p^\ast, \widehat{K}) \le d(p^\ast, z) + d(z, \widehat{K}) \le \frac{\epsilon}{3} + \frac{\epsilon}{3} <\epsilon,
$$
 contradicting \eqref{e.acontradicao}. 
This ends the proof of the lemma.
\end{proof}

 \subsection{Proof of Theorem \ref{semifractal}}
 
 The first item of the theorem follows from the next two assertions that we will prove below:

\begin{itemize}
\item[(1)] 
If $K$ is a non-empty compact fixed point of $\mathcal B$, then  $\overline{A_\mathrm{tar}}\subset K$;
 \item [(2)]
 $\overline{A_\mathrm{tar}}$ is a fixed point 
 of the $\mathrm{BH}$-operator $\mathcal{B}$.
\end{itemize} 
 
 To prove the first assertion,
 take any point $p\in A_\mathrm{tar}$. It follows from the definition of $A_{\mathrm{tar}}$, see Remark \ref{EDtarget}, that 
 there is a sequence $\omega$ such that 
 \begin{equation}\label{e.asinequation}
 \{p\}=\bigcap_{n\geq 0} f_{\omega_{0}}\circ\cdots\circ f_{\omega_{n}}(X)
 \supset \bigcap_{n\geq 0} f_{\omega_{0}}\circ\cdots\circ f_{\omega_{n}}(K).
 \end{equation}
 Since the last intersection is non-empty and contained in $K$, it follows that $p\in K$. 
 As this holds for every point $p\in A_\mathrm{tar}$, we have that
 $ A_\mathrm{tar}\subset K$. As $K$ is closed this implies that $\overline{A_\mathrm{tar}}\subset K$.

 It remains 
to prove the second assertion,  $\mathcal B(\overline{A_\mathrm{tar}})=\overline{A_\mathrm{tar}}$. Note that 
for $p$ as in \eqref{e.asinequation} and  every $i=1,\dots,k$ it holds
\begin{equation}
\{f_{i}(p)\}=
\bigcap_{n\geq 0} f_{i}\circ f_{\omega_{0}}\circ\cdots\circ f_{\omega_{n}}(X).
\end{equation}
This implies that $\mathcal{B}(A_{\mathrm{tar}})\subset{A_{\mathrm{tar}}}$.  Hence, by continuity of the maps $f_i$,
$\mathcal{B}(\overline{A_{\mathrm{tar}}})\subset\overline{A_{\mathrm{tar}}}$. By 
definition of the set  $(\overline{A_{\mathrm{tar}}})^*$ in \eqref{e.conjestrela}, it is clear that $(\overline{A_{\mathrm{tar}}})^*\subset \overline{A_{\mathrm{tar}}}$. 
By Lemma~\ref{l.p.existence}, $(\overline{A_{\mathrm{tar}}})^*$ is a fixed point of $\mathcal{B}$. 
Since, by the first assertion, the set $\overline{A_{\mathrm{tar}}}$ is contained 
in every fixed point of $\mathcal{B}$, it follows that $\overline{A_{\mathrm{tar}}}\subset (\overline{A_{\mathrm{tar}}})^*$ and
hence  $\overline A_{\mathrm{tar}}= (\overline{A_{\mathrm{tar}}})^*$ .

To prove the second item of the theorem, consider any non-empty compact subset $K\subset \overline{A_{\mathrm{tar}}}$.
It follows from \eqref{e.asinequation} that $\mathcal{B}(K)\subset  \overline{A_{\mathrm{tar}}}$ and hence 
 $\mathcal{B}^{n}(K)\subset  \overline{A_{\mathrm{tar}}}$ for every $n\geq 0$. In particular, 
 recalling the definition of $h_s$ in \eqref{e.HD}, it follows
 $$
 \lim_{n\to \infty} h_{s}(\mathcal{B}^{n}(K),\overline{A_{\mathrm{tar}}})=0.
 $$ 
 On the other hand, it follows from Lemma \ref{dista} that $
 \lim_{n\to \infty} h_{s}(\overline{A_{\mathrm{tar}}},\mathcal{B}^{n}(K))=0
 $
 and hence 
 $$
 \lim_{n\to \infty} d_{H}(\mathcal{B}^{n}(K),\overline{A_{\mathrm{tar}}})=0.
 $$

\subsection{Proof of Theorem~\ref{mt.local}}
\label{ss.mtlocal}

If there are no strict attractors we are done. Otherwise, assume that there is 
a strict attractor $K$. Since $K$ is a fixed point of $\mathcal{B}$, the first item of Theorem \ref{semifractal}
 implies that
$\overline{A_{\mathrm{tar}}}\subset K$. 
 By definition of a strict attractor, the set  $K$ attracts every compact set 
in a neighbourhood of it, and thus $\overline{A_{\mathrm{tar}}}$
is attracted by $K$. Therefore $\overline{A_{\mathrm{tar}}}=K$, proving the first item. 

Since every strict attractor is a Conley attractor, to prove the second item of the theorem it is enough to see that if  $\overline{A_{\mathrm{tar}}}$ is a 
 Conley attractor then $\overline{A_{\mathrm{tar}}}$ is a strict attractor. 
 Thus, we now assume that $\overline{A_{\mathrm{tar}}}$ is a Conley attractor and hence it has an
  open neighbourhood  $U$  such that 
$\lim_{n\to \infty} \mathcal{B}^{n}(\overline{U})=\overline{A_{\mathrm{tar}}}$.
To prove that
$\overline{A_{\mathrm{tar}}}$ is a strict attractor it is enough  to check that
for every compact set 
$K\in \mathscr{H}(\overline{U})$ it holds
$\lim_{n\to \infty}\mathcal{B}^{n}(K)=\overline{A_{\mathrm{tar}}}$. 
In other words,
we need to see that
for every $\epsilon >0$ there is $n_{0}\in \mathbb{N}$ such that 
for every $n\geq n_{0}$ we have
 \begin{equation}
 \label{e.toproveremains}
 d_{H}(\overline{A_{\mathrm{tar}}},\mathcal{B}^{n}(K))=\max\{ h_{s}(\overline{A_{\mathrm{tar}}},\mathcal{B}^{n}(K))
 ,  h_{s}(\mathcal{B}^{n}(K),\overline{A_{\mathrm{tar}}})\}
 \leq\epsilon.
 \end{equation}

Since $\lim_{n\to \infty}\mathcal{B}^{n}(\overline{U})=
\overline{A_{\mathrm{tar}}}$,  there is $n_{0}$
such that for every  $n\geq n_{0}$
we have
$$
 h_{s}(\mathcal{B}^{n}(\overline{U}),\overline{A_{\mathrm{tar}}})
\leq\epsilon, 
$$
which implies that for every $n\ge n_0$,
$$
 h_{s}(\mathcal{B}^{n}(K),\overline{A_{\mathrm{tar}}})\leq 
 h_{s}(\mathcal{B}^{n}(\overline{U}),\overline{A_{\mathrm{tar}}})
\leq\epsilon.
$$
Hence,  to prove \eqref{e.toproveremains} it remains to see that $
h_{s}(\overline{A_{\mathrm{tar}}},\mathcal{B}^{n}(K))
\le \epsilon $ for
every $n$ sufficiently large. This is exactly the content of Lemma \ref{dista}.

The proof of the theorem is now complete.
\hfil \qed

\begin{schol}\label{sc.atracaofatal2}
If $U$ is a neighbourhood of  $\overline{A_{\mathrm{tar}}}$ such that
$\lim_{n\to \infty}\mathcal{B}^n(\overline{U}) =\overline{A_{\mathrm{tar}}}$, then every compact subset
of $\overline{U}$ also satisfies $\lim_{n\to \infty}\mathcal{B}^n(K) =  \overline{A_{\mathrm{tar}}}$.
\end{schol}

 \subsection{Proof of Theorem~\ref{mt.att}} 
 \label{ss.attractor}
Note that by hypothesis
 the set $A_{\mathrm{tar}}$ is non-empty.
 We need to prove the equivalence of
the  following three assertions: 
\begin{enumerate}
 \item 
$\overline{A_\mathrm{tar}}=X^*$;
 \item
the $\mathrm{BH}$-operator  has a unique fixed point;
\item
$X^{*}$ is a global attractor
(a strict attractor whose basin is the whole space).
\end{enumerate}

The implication $(1)\Rightarrow (2)$ follows
from the first item of Theorem \ref{semifractal}, which in this case means that
``the minimum fixed point $\overline{A_\mathrm{tar}}$ and
the maximum fixed point $X^*$ are the same''. The converse implication $(1)\Leftarrow (2)$ follows from the
fact that $\overline{A_\mathrm{tar}}$  and $X^*$ are fixed points of $\mathcal{B}$.
Thus we get $(1)\Leftrightarrow (2)$

The implication $(3) \Rightarrow (2)$ is immediate.

To prove 
 $(1)\Rightarrow (3)$ note that, by Claim~\ref{c.l.semnome}, $X^{*}=\lim_{n\to\infty}\mathcal{B}^{n}(X)$ and thus $X^*$ is a Conley attractor. Since, by hypothesis, $\overline{A_{\mathrm{tar}}}=X^{*}$, we 
 have that $\overline{A_{\mathrm{tar}}}$ is a Conley attractor and hence it follows from Scholium ~\ref{sc.atracaofatal2} that $\overline{A_{\mathrm{tar}}}$ is
 a strict attractor whose basin is the whole space.
 \hfil \qed

\subsection{Proof of Theorem~\ref{mt.jogodocaos}} 
\label{ss.mtjogo}

By hypothesis, the set  $\overline{A_{\mathrm{tar}}}$ is a stable fixed point of the 
$\mathrm{BH}$-operator.
Fixed a point $x$ and a sequence $\omega$, consider the nested sequences of sets
$$
Y_\ell (x,\omega)\eqdef \{f_{\omega}^{n}(x)\colon n\geq \ell\}.
$$
 In view of Claim \ref{c.l.semnome}, to prove the theorem it is enough to see 
 that for every disjunctive sequence $\omega$ and every point $x$
it holds
\begin{equation}
\label{e.provadoteorema}
\overline{A_{\mathrm{tar}}}=\bigcap_{\ell \geq 0} \overline{\{f_{\omega}^{n}(x)\colon n\geq \ell\}}=
\bigcap_{\ell \geq 0}  \overline{Y_\ell (x,\omega)}.
\end{equation}
To simplify notation, fixed $\omega$ and $x$ as above, write 
$Y_\ell=Y_\ell (x,\omega)$. 
 
 \medskip
 

\noindent{{\emph{Proof of  the inclusion ``$\subset$''.}} 
Take any point $p\in A_{\mathrm{tar}}$ and fix $\ell\geq 0$. We need to see 
that for every  neighbourhood $V$ of $p$ it holds
\begin{equation}\label{e.inclusion0205}
V \cap Y_\ell\ne \emptyset.
\end{equation}
By definition of $A_{\mathrm{tar}}$, there is an arbitrarily large finite sequence 
$c_{0}\dots c_{r}$ such that 
\begin{equation}\label{e.abril16}
f_{c_{r}}\circ \dots \circ f_{c_{0}}(X)\subset V.
\end{equation}
 We can  assume that $r\geq \ell$. Since
  $\omega$ has dense orbit there is $m_{1}$ such that 
  $$
  \omega_{m_{1}}=c_{0},\, \omega_{m_{1}+1}=c_{1},\,
  \dots,\,\omega_{m_{1}+r}=c_{r}.
 $$ 
   Therefore, from \eqref{e.abril16} it follows 
   $$
  f_{\omega}^{m_{1}+r+1}(x)\in V.
  $$
 Since $m_{1}+r\geq \ell$ we have that
  $ f_{\omega}^{m_{1}+r+1}(x)\in Y_\ell$ and hence
  $V\cap Y_\ell 
 \neq \emptyset$, proving 
 \eqref{e.inclusion0205} and hence the inclusion `$\subset$''.
 
\medskip

\noindent{{\emph{Proof of  the inclusion ``$\supset$''.}} 
  Take any neighbourhood $V$ of $\overline{A_{\mathrm{tar}}}$. Since 
$\overline{A_{\mathrm{tar}}}$ is stable it has a neighbourhood $V_{0}\subset V$
such that $\mathcal{B}^{n}(V_{0})\subset V$ for every $n\geq 0$.
By definition of $A_{\mathrm{tar}}$, there is a sequence $c_{0}\dots c_{r}$ with
$f_{c_{r}}\circ \dots \circ f_{c_{0}}(X)\subset V_0$. As above, since
  $\omega$ is disjunctive, there is $m_{1}$ with
 $  \omega_{m_{1}}=c_{0},\, \omega_{m_{1}+1}=c_{1},\,
  \dots,\,\omega_{m_{1}+r}=c_{r}.$  
  Taking $n_0=m_1+r$,
this implies that
$f_{\omega}^{n_{0}+1}(x)\in V_{0}$.
From $\mathcal{B}^{n}(V_{0})\subset V$ for every $n\geq 0$ it follows
that
$Y_{n_0} \subset V$, and thus
 $\overline{Y_{n_0}}\subset \overline{V}$.
 As the sequence of sets 
 $(\overline{Y_\ell})$
 is nested,  we have that 
 $
 \bigcap_{\ell \geq 0} \overline{Y_\ell}\subset \overline{V}.
 $
 Since this holds for every neighbourhood $V$ of $\overline{A_{\mathrm{tar}}}$, we conclude that 
  $$
 \bigcap_{\ell \geq 0} \overline{Y_\ell}
 \subset \overline{A_{\mathrm{tar}}} ,
 $$
 proving the inclusion ``$\supset$''. 
 
 We have proved the equality in \eqref{e.provadoteorema} and hence the theorem.
 \hfil 
\qed

\section{Examples}
\label{s.examples}
In this section, we present some examples that adds a bit more to our
understanding on the target set and also illustrate the hypotheses in the theorems.
The first example, Example \ref{ex.nonregular}, is the most important one of this section
and shows a non-regular IFS whose target
is non-empty.
Before going to the details, let us say some general words about our  constructions and examples. 

Note that requiring the target set 
to be non-empty is equivalent to assume that $S_{\mathrm{wh}}\neq \emptyset$. To construct 
an IFS$(f_{1},\dots,f_{k})$ defined on a compact metric space $X$  with $S_{\mathrm{wh}}\neq \emptyset$, a natural
(and also the simplest) procedure is  asking
 some map to be a contraction, or more generally, asking that for
some $\ell\in\{1, \dots, k\}$, there is $p$ such that
\begin{equation}\label{single}
\bigcap_{i\ge 1} f_{\ell}^{i}(X)=\{p\}.
\end{equation}
However, since the single map $f_{\ell}$ 
has a fixed point (the point $p$) that is a global attractor, such a hypothesis leads to a regular IFS (recall the definition in Section \ref{regularifs}).  We observe that 
this is  precisely the scenario that one finds when studying contracting on average IFSs 
on 
compact metric spaces. 
In Example \ref{ex.nonregular}, we show how to construct IFSs with 
$S_{\mathrm{wh}}\neq \emptyset$ without assuming that one of the maps has a 
fixed point that is a
global attractor 
 as in \eqref{single}. The idea is to use Theorem 4.1 in \cite{DiazMatias} claiming  
that for IFSs on compacts subsets of $\mathbb{R}^{m}$ satisfying a simple topological condition it holds 
$\mathbb{P}(S_{\mathrm{wh}})=1$, for every product measure $\mathbb{P}=\nu^{\mathbb{N}}$, where 
$\nu$ is any probability measure on $\{1,\dots, k\}$ with $\nu(\{i\})>0$ for every $i$. In particular, for the class of IFSs  in  \cite[Theorem 4.1]{DiazMatias} it holds $S_{\mathrm{wh}}\neq \emptyset$. In Example \ref{ex.nonregular} the IFS is defined on $[0,1]$, 
but using the same idea we can also perform examples in higher dimensions.

\medskip

In  Example \ref{example1}, we construct an IFS satisfying the hypothesis in Theorem \ref{mt.jogodocaos} such that  the closure of the target set is not a Conley attractor.

%
%

We now discuss Examples~\ref{ex.ultimahora} and \ref{ex.outromais} togheter.
First recall Remark \ref{EDtarget} stating that for weakly hyperbolic IFSs the target set is equal to the maximum fixed point,
i.e., $A_{\mathrm{tar}}=X^*$.
Example \ref{ex.ultimahora} shows that this equality can also happen for 
non-weakly hyperbolic IFSs. This is an example of a ``genuinely'' non-hyperbolic IFS 
 whose BH-operator has a global attractor fixed point,
  recall Theorem \ref{mt.att}.
 However,  Example \ref{ex.outromais} shows that the equality 
$A_{\mathrm{tar}}=X^*$
 is not a necessary condition 
 for the $\mathrm{BH}$-operator having a global attractor fixed point. In fact,
 Example~\ref{ex.outromais} gives an IFS where 
 $\overline{A_{\mathrm{tar}}}$ coincides with the maximum fixed point but 
 $A_{\mathrm{tar}}\subsetneq \overline{A_{\mathrm{tar}}}$. 
 In other words, in Theorem \ref{mt.att} it is indispensable considering the closure on the target set.

 \medskip
 In Examples \ref{ex.outromais} and \ref{ex.ultimahora}, we consider IFSs consisting of two maps on the interval $[0,1]$.
 We observe that to construct IFSs where the target set has the properties in these examples
 is much  more easier if we allow the IFS to have more than 2 maps.
 However, IFSs consisting of two maps  whose  
  target sets have the properties described above can be found
  in the constructions of two well known class of attractors (of a dynamical system), the so-called \emph{bony attractors} 
 and  \emph{porcupine-like horseshoes}, see \cite{Yuri,DiGe}. Hence, we also took the opportunity in this section to describe the target set of the IFSs used in the construction of these  attractors.
 
 \medskip
 
We close this section with the quite simple Example~\ref{ex.without} showing that, in general, IFSs may not have a minimum fixed point.

\begin{example}[A non-regular IFS with $S_{\mathrm{wh}}\neq \emptyset$]
\label{ex.nonregular}
{\emph{
Consider an $\mathrm{IFS}(f_1,f_2)$ such that $f_1$ and $f_2$ are continuous and injective maps defined on $[0,1]$
satisfying:
\begin{itemize}
\item[(1)]
 $f_{1}$ has  exactly
two fixed points $0$ and $1$, where  $0$ is a repeller 
and $1$ is an attractor;
\item[(2)]
$f_2$ 
has exactly three fixed points $p_{1}<p_{2}<p_{3}$,
where $p_1$ and $p_3$ are attractors and $p_2$ is a  repeller;
\item[(3)]
$f_{2}([0,1])=[\alpha,\beta]\subset (0,1)$ and
$p_1< f_1(p_1)< \beta$;
\item[(4)]
there is $c \in (\alpha,p_1)$ such that $f_1(c)=f_2(c)$ and $f_1(x)\ne f_2(x)$ for all $x\ne c$.
\end{itemize}
These conditions are depicted in Figure \ref{f.nonregular}.
We claim that:
\begin{itemize}
\item[(i)]
$\mathrm{IFS}(f_1,f_2)$ is non-regular;
\item[(ii)]
$
S_{\mathrm{wh}}\neq \emptyset$.
\end{itemize}
\begin{figure}[h!]
\centering
\begin{tikzpicture}[xscale=3,yscale=3]
\draw[->] (0,0)-- (1,0);
\draw[->] (0,0)--(0,1);
\draw[-]  (0,0.1)--(0.33,0.25);
\draw[-]    (0.33,0.25)--(0.45,0.52);
\draw[-] (0.45,0.52)--(1,0.7);
\draw[dashed] (0,0)--(1,1);
  \draw[green,smooth,samples=100,domain=0.0:1] plot(\x,{-\x*\x+2*\x});
 \node[scale=0.8,left] at (0.4,0.7) {$f_{1}$};
  \node[scale=0.8,below] at (0.8,0.6) {$f_{2}$};
\draw[dashed] (0,1)--(1,1);
\draw[dashed] (1,0)--(1,1);
\end{tikzpicture}
\caption{A non-regular IFS with $S_{wh}\neq \emptyset$.}
\label{f.nonregular}
\end{figure}
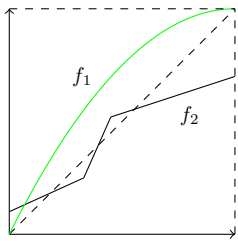}}
{\emph{Recall first the definition of a
regular IFS in Section \ref{regularifs}.
To prove (i), first note that neither $\mathrm{IFS}(f_1)$ or $\mathrm{IFS}(f_2)$ has 
a global attractor.
Thus, it remains to see that
$\mathrm{IFS}(f_{1},f_{2})$ has not a global attractor. We prove this assertion by showing 
that the $\mathrm{BH}$-operator $\mathcal{B}$ associated with
$\mathrm{IFS}(f_1,f_2)$ has at least two different fixed points (which prevents regularity). Indeed, we have that
$[0,1]$ and $[p_{1},1]$ are fixed points of $\mathcal{B}$. It is clear that
$[0,1]$ is a fixed point. To see that  $[p_{1},1]$ is a fixed point, note that 
 $p_1<f_1(p_1)<\beta$, hence 
$$
\mathcal{B}([p_1,1])=
f_2([p_1,1]) \cup  f_1([p_1,1])= [p_1,\beta] \cup [f_1(p_1),1]= [p_1,1].
$$
This implies that 
$\mathrm{IFS}(f_1,f_2)$ is non-regular.}}

\emph{
To prove (ii), note that
$1$ is an attracting fixed point of $f_1$ and $f_{2}([0,1])\subset(0,1)$. Hence we get the \emph{splitting property} $f_{1}^{n}\circ f_{2}([0,1])\cap f_{2}([0,1])=\emptyset$
 for $n$ sufficiently large.
Now  Theorem 4.1 in \cite{DiazMatias}, implies that $S_{\mathrm{wh}}\neq \emptyset$.}

\end{example}

 \begin{example}[An IFS without 
strict attractors and such that $\overline{A_{\mathrm{tar}}}$ is stable]
\label{example1}{\emph{
Consider maps $f_{1},\, f_2\colon[0,2]\to [0,2]$ 
defined by (see Figure \ref{f.nconley}):
\begin{itemize}
\item[(1)]
$f_{1}(x)=\frac{1}{3}x$;
\item[(2)]
$f_{2}(x)=\frac{1}{3}x+\frac{2}{3}$ for $x\in[0,1]$ and $f_{2}(x)=x$ for 
$x\in [1,2]$.
\end{itemize}
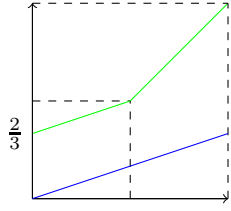
\begin{figure}[!h]
  \centering
  \begin{tikzpicture}[xscale=1.3,yscale=1.3]
      \draw[->] (1,0) -- (2,0);
    \draw[->] (0,1) -- (0,2);
    \draw[green,smooth,samples=100,domain=1.0:2.0] plot(\x,{\x});
     \draw[-] (0,0) -- (1,0);
      \draw[-] (0,0) -- (0,1);
       \draw[dashed] (2,0) -- (2,2);
       \draw[dashed] (0,2) -- (2,2);
        \draw[dashed] (0,1) -- (1,1); 
         \draw[dashed] (1,0) -- (1,1);
       \draw[green,smooth,samples=100,domain=0.0:1.0] plot(\x,{0.33333*\x+0.66666});
     \draw[-] (0,0) -- (2,0);
      \draw[-] (0,0) -- (0,2);
      \draw[blue,smooth,samples=100,domain=0.0:2.0] plot(\x,{0.33333*\x});
       \node[scale=1,left] at (0,0.66666) {$\frac{2}{3}$};
       \end{tikzpicture}
     \caption{The set $\overline{A_{\mathrm{tar}}}$ is Lyapunov stable and is not a Conley attractor}
    \label{f.nconley}
     \end{figure}
  We claim that:
  \begin{itemize}
  \item [(i)]$S_{\mathrm{wh}}\ne \emptyset$;
  \item [(ii)] $\overline{A_{\mathrm{tar}}}$ is not a Conley attractor;
  \item[(iii)] $\overline{A_{\mathrm{tar}}}$ is  stable.
  \end{itemize}
}}
\emph{
Item (i) follows observing that
$f_{1}$ is a contraction}.

{\emph{
To prove (ii),  consider the auxiliary
$\mathrm{IFS}(g_{1},g_{2})$
where $g_{1}=f_{1|[0,1]}$ and $g_{2}=f_{2|[0,1]}$. Let $\mathcal{C}$ be the standard ternary Cantor set in the interval $[0,1]$.
We claim that $\overline{A_{\mathrm{tar}}}=\mathcal{C}$.
Note 
that $\mathcal{C}$ is the Hutchinson attractor of  $\mathrm{IFS}(g_{1},g_{2})$ (see, for instance,
Example 1 in  \cite[Section 3.3]{Hu}). 
In particular, the set 
$\mathcal{C}$ is the unique fixed point of the $\mathrm{BH}$-operator 
$\mathcal{B}$ of  $\mathrm{IFS}(f_{1},f_{2})$ contained in $[0,1]$. Since $[0,1]$
is $\mathcal{B}$-invariant, by Lemma~\ref{l.p.existence} and minimality 
(item (1) in Theorem~\ref{semifractal})  we have 
$\overline{A_{\mathrm{tar}}}\subset [0,1]^* \subset[0,1]$
and therefore $\overline{A_{\mathrm{tar}}}=\mathcal{C}$, proving the claim.
To see $\mathcal{C}$ is not a Conley
attractor, just note that every open neighbourhood of $\mathcal{C}$ necessarily contains an interval 
of the form $[1,\delta)$. Since $f_{2}(x)=x$ for all $x\in [1,\delta)$ the assertion
follows.}}

{\em{
To see (iii), just note that $f_{1},f_{2}$ are non-expanding maps.}}
\end{example}

In the next example, we consider 
 the underlying IFS of the bony attractor in \cite{Yuri}.

\begin{example}[A non-weakly hyperbolic IFS on $\mbox{[0,1]}$ with $A_{\mathrm{tar}}=\mbox{[0,1]}$]
\label{ex.ultimahora}
\emph{Consider the $\mathrm{IFS}(f_{1},f_{2})$
whose maps $f_1$ and $f_2$ are
 defined on $[0,1]$ as follows  (see Figure~\ref{bony}):
 \begin{itemize}
 \item[(1)]
 $f_{1}$ is the piecewise-linear map
with ``vertices'' $(0, 0)$, $(0.6, 0.2)$, and $(1, 0.8)$;
\item[(2)]
$f_{2}$ is the piecewise-linear map with ``vertices''
$(0, 0.15)$, $(0.4, 0.8)$, and $(1, 1)$.
\end{itemize}
As observed in \cite{Yuri}, these maps have the following properties:
\begin{itemize}
\item[(3)]
$f_{1}\circ f_{2}$ has a repelling fixed point;
\item[(4)]
the compositions $f_{1}^3$, $f_{1}^{2}\circ f_{2}$, $f_{2}^{2}\circ f_{1},$ and $f_{2}^{5}$ are uniform contractions
and the union of their images is $[0,1]$.
\end{itemize}
We claim that:
\begin{itemize}
\item[(i)]The
$\mathrm{IFS}(f_{1},f_{2})$ is non-weakly hyperbolic;
\item[(ii)]$A_{\mathrm{tar}}=[0,1]$.
\end{itemize}  
}

\begin{figure}[h!]
\centering
\begin{tikzpicture}[xscale=3,yscale=3]
 \node[scale=0.6,left] at (0.3,0.75) {$f_{2}$};
  \node[scale=0.6,below] at (0.7,0.6) {$f_{1}$};
\draw[->] (0,0)-- (1,0);
\draw[->] (0,0)--(0,1);
\draw[red,-]  (0,0)--(0.6,0.2);
\draw[red,-]  (0.6,0.2)--(1,0.75);
\draw[blue,-]  (0,0.15)--(0.4,0.8);
\draw[blue,-]  (0.4,0.8)--(1,1);
\draw[dashed] (0,1)--(1,1);
\draw[dashed] (0,0)--(1,1);
\draw[dashed] (1,0)--(1,1);
\end{tikzpicture}
\caption{The underlying IFS of a bony attractor}
\label{bony}
\end{figure}
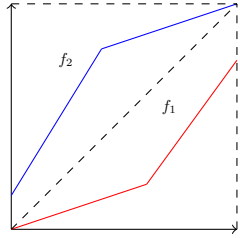

\emph{
The fact that $f_{1}\circ f_{2}$ has  a repelling fixed 
point implies (i): the  periodic sequence $\overline{12}$ does not belong to $S_{\mathrm{wh}}$.}

{\em{
To prove (ii),
consider the finite set of words corresponding to the compositions in item $(4)$
$$
W=\{111,112,221, 22222\}
$$ 
and 
 let $E_{W}$ be the subset of $\Sigma_{2}$ consisting 
 of sequences $\omega$ that are a concatenation 
 of words of $W$.
It follows from the classical Hutchinson theory for hyperbolic IFSs that $E_{W}\subset S_{\mathrm{wh}}$ and 
$\pi(E_{W})=[0,1]$, where $\pi$ is the coding map (this follows from condition(4)).
Thus $A_{\mathrm{tar}}=\pi(S_{\mathrm{wh}})=[0,1]$.}}
\end{example}

In the next example we consider 
 the underlying IFS of the porcupine-like
horseshoes in \cite{DiGe} and translate the construction in \cite[page 12]{DiGeRams} to
our context.

\begin{example}[{$A_{\mathrm{tar}}\subsetneq \overline{A_{\mathrm{tar}}}$ and $\overline{A_{\mathrm{tar}}}$ is the whole space}]
\label{ex.outromais}
{\emph{Consider the $\mathrm{IFS}(f_{1},f_{2})$
whose maps $f_1$ and $f_2$ are
 defined on $[0,1]$ as follows (see Figure~\ref{Porcupine}):
\begin{itemize}
\item[(1)]
$f_{1}(x)=\lambda \, (1-x)$, $\lambda\in (0,1)$;
\item[(2)]
$f_{2}$ is an injective continuous map with exactly two fixed points,
the repelling fixed point $0$
 and the attracting 
fixed point $1$;
\item[(3)]
$f_{2}$ is a uniform contraction 
 on $[f_{2}^{-1}(\lambda),1]$. 
 \end{itemize}
 We claim that:
 \begin{itemize}
 \item[(i)]
 $\overline{A_{\mathrm{tar}}}=[0,1]$; 
 \item[(ii)]
 $A_{\mathrm{tar}}\subsetneq \overline{A_{\mathrm{tar}}}$.
 \end{itemize}
 }}
 
{\emph{To prove (i), we start by observing that $\lambda \in \overline{A_{\mathrm{tar}}}$.  To see why this is so,
 take any open neighbourhood $V\subset (0,1)$ of $\lambda$ and note that $f_{1}^{-1}(V)$ is a neighbourhood 
 of $0$. 
 Consider the fixed point
  $p=\frac{\lambda}{1+\lambda}\in (0,1)$ 
 of $f_{1}$ and  note $\bar 1\in S_{\mathrm{wh}}$ and hence $p\in A_{\mathrm{tar}}$. 
 Since $f_{2}^{n}(p)\to 1$ as $n\to \infty$ and $f_{1}(1)=0$, there is $\ell$ such that 
 $f_{1}\circ f_{2}^{\ell}(p)\in f_{1}^{-1}(V)$. Hence $f_{1}^{2}\circ f_{2}^{\ell}(p)\in V$.
 The invariance of $A_{\mathrm{tar}}$ implies  that $A_{\mathrm{tar}}\cap V\neq \emptyset$.
 Since this holds for every neighbourhood  $V$ of $\lambda$ we get $\lambda \in \overline{A_{\mathrm{tar}}}$.
 }}

{\emph{
 We now  prove that $A_{\mathrm{tar}}$ is dense in $[0,1]$:  for every open interval $J\subset (0,1)$
 it holds $J\cap A_{\mathrm{tar}}\neq \emptyset$. If $\lambda\in J$ we are done. Otherwise $\lambda\notin J$
 and
 either $J\subset (\lambda,1]\eqdef I_{2}$ or $J\subset [0,\lambda)\eqdef I_{1}$. 
We claim that 
\begin{equation}
\label{e.forsome}
\lambda \in f_{\omega_{m}}^{-1}\circ\dots\circ f_{\omega_{0}}^{-1}(J),
\quad
\mbox{for some $\omega_{0}\dots \omega_{m}$}.
\end{equation}
Since $\lambda \in \overline{A_{\mathrm{tar}}}\cap f_{\omega_{m}}^{-1}\circ\dots\circ f_{\omega_{0}}^{-1}(J)$,
the invariance of $A_{\mathrm{tar}}$ implies that $J\cap A_{\mathrm{tar}}\neq \emptyset$.
Thus it remains to prove \eqref{e.forsome}.
For that let $\omega_{0}=i$ if $J\subset I_{i}$ and define recursively 
$\omega_{\ell+1}=i$ if $f_{\omega_{\ell}}^{-1}\circ\dots\circ f_{\omega_{0}}^{-1}(J)\subset I_{i}$.
Note that if at some step
$f_{\omega_{\ell}}^{-1}\circ\dots\circ f_{\omega_{0}}^{-1}(J)$ is not contained neither in $I_1$ nor in $I_2$ we have that
 $\lambda \in f_{\omega_{\ell}}^{-1}\circ\dots\circ f_{\omega_{0}}^{-1}(J)$ 
and  we are done.
As $f_{2}^{-1}$ is a uniform expansion on $I_2$ and $f_{1}^{-1}$ is a uniform 
expansion on $I_1$, the size of the intervals 
$f_{\omega_{\ell}}^{-1}\circ\dots\circ f_{\omega_{0}}^{-1}(J)$
increases exponentially and 
the recursion stops after a finitely many steps:
there is $m$ such that $\lambda\in f_{\omega_{m}}^{-1}\circ\dots\circ f_{\omega_{0}}^{-1}(J)$,
proving \eqref{e.forsome}.
}}

{\emph{To prove (ii), it is enough to see that $1\notin A_{\mathrm{tar}}$. This follows noting that 
$\bar 2 \not\in S_{\mathrm{wh}}$ and that
$1\notin f_{\omega_{0}}\circ\dots\circ f_{\omega_{n}}([0,1])$ for
any finite 
sequence $\omega_{0}\dots \omega_{n}$ such that $\omega_{i}=1$ for some $i$. } }
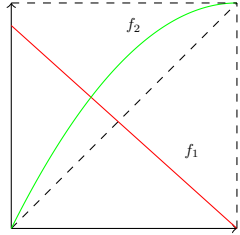
\begin{figure}[h!]
\centering
\begin{tikzpicture}[xscale=3,yscale=3]
\draw[->] (0,0)-- (1,0);
\draw[->] (0,0)--(0,1);
\draw[red,-] (0,0.9)--(1,0);
\draw[dashed] (0,0)--(1,1);
  \draw[green,smooth,samples=100,domain=0.0:1] plot(\x,{-\x*\x+2*\x});
 \node[scale=0.6,left] at (0.6,0.9) {$f_{2}$};
  \node[scale=0.6,below] at (0.8,0.4) {$f_{1}$};
\draw[dashed] (0,1)--(1,1);
\draw[dashed] (1,0)--(1,1);
\end{tikzpicture}
\caption{The underlying IFS of a porcupine-like horseshoe}
\label{Porcupine}
\end{figure}
\end{example}

\begin{example}[An IFS without a minimum fixed point]
\label{ex.without}{\emph{
Consider
the $\mathrm{IFS}(f_{1},f_{2})$ where  $f_1,f_2\colon [0,1]\to [0,1]$ are defined by $f_{1}(x)=x$
and $f_{2}(x)=1-x$. For each $x\in [0,1]$, the set $\{x,1-x\}$ is a fixed 
point of $\mathcal{B}$ which is also minimal.
Obviously,  there is no  
fixed point contained in all fixed points and hence there is no a minimum fixed point.
Note that, in this example,  we have $S_{\mathrm{wh}}= \emptyset$.
}}
\end{example}
%
%
%
%

%
%
%
%
%
%
%

%

%




\bibliographystyle{acm}
\bibliography{references}

\end{document}